\documentclass[a4paper,12pt,oneside]{article}
\usepackage{amssymb,amsmath,amsthm,amsfonts,eurosym,geometry,ulem,graphicx,caption,titlesec,float}
\usepackage{xcolor,setspace,sectsty,comment,pdflscape,subfigure,array,enumitem}
\usepackage{booktabs}
\usepackage[misc]{ifsym}
\usepackage[colorlinks=true,linkcolor=blue]{hyperref}%
\hypersetup{urlcolor=blue, citecolor=blue}
\usepackage{hyperref}
\usepackage{leftidx}
\DeclareMathOperator{\prox}{prox}
\usepackage{geometry}
\usepackage{titling}
\pretitle{\vspace*{-5em}\leftline{\footnotesize {\bf Submitted on \today
		}}\vspace*{-0.23em}
	\vspace*{5em}
	\begin{center}\fontsize{20pt}{22pt}\selectfont\bfseries}% co font Title
	\posttitle{\par\end{center}\vskip 2ex}
\usepackage{graphicx}

\usepackage[misc]{ifsym}
\newtheorem{theorem}{Theorem}[section]
\newtheorem{definition}{Definition}[section]
\newtheorem{lemma}{Lemma}[section]
\newtheorem{corollary}{Corollary}[section]

\newtheorem{example}{Example}[section]

\newtheorem{remark}{Remark}[section]
\usepackage[section]{algorithm}
\usepackage{algorithm,algpseudocode,float}

\usepackage{fancyhdr}
\newcommand\authors{\small\bfseries Golden ratio algorithms for solving equilibrium problems in Hilbert spaces\hfill}
\newcommand\shorttitle{\small\bfseries\hfill N.T. Vinh}
\makeatletter
\renewcommand{\@seccntformat}[1]{\csname the#1\endcsname.\ }
\makeatother
\begin{document}

\title{Golden ratio algorithms for solving equilibrium problems in Hilbert spaces}
%http://www.birs.ca/events/2017/5-day-workshops/17w5030/videos/watch/201709181114-Malitsky.html
\author{Nguyen The Vinh$^1$}
%\date{\footnotesize Department of Mathematics, University of Transport and Communications, Hanoi, Vietnam
%\vspace*{-3.5em}}
\date{\vspace*{-4em}}
%\date{\small November 17, 2017. Revised: January 30, 2018.\vspace*{-3.5em} }
\renewcommand{\thefootnote}{}
\renewcommand{\thefootnote}{}
%\footnotetext{$^*$Corresponding author.\\ \vspace*{-0.02cm}
%\hskip 0.01cm	E-mail addresses: thevinhbn@utc.edu.vn (N.T. Vinh), avivg@braude.ac.il (A. Gibali).\\ \vspace*{0.1cm}
%\hskip 0.01cm	Received xxxxxxxx xx, 2018.;\ \ Revised xxxxxxxx xx, 2018. }
\footnotetext{\hskip -0.61cm This research is funded by Vietnam National Foundation for Science and Technology Development (NAFOSTED) under Grant No.101.01-2017.08.\vspace*{-0.65em}\\ \noindent
\rule{0.5\textwidth}{0.1pt}
	\begin{enumerate}[label=(\alph*),leftmargin=1.65\parindent]
		\item[\Letter] Nguyen The Vinh\\
		thevinhbn@utc.edu.vn
		\item[$^1$\ ] Department of Mathematics, University of Transport and Communications, Hanoi, Vietnam%\\ %\vspace*{-0.5cm}
%		\item[$^2$\ ]Department of Mathematics, University of Transport and Communications, Hanoi, Vietnam. E-mail addresses: thevinhbn@utc.edu.vn (N.T. Vinh),
	\end{enumerate}
}

\maketitle
\thispagestyle{empty}

\begin{abstract}
\noindent {\bf Abstract.} 
In this paper, we design a new iterative algorithm for solving pseudomonotone equilibrium problems in real Hilbert spaces. The advantage of our algorithm is that it requires only one strongly convex programming problem at each iteration.  Under suitable conditions we establish the strong and weak convergence of the proposed algorithm. The results presented in the paper extend and improve some recent results in the literature. The performances and comparisons with some existing methods are presented through numerical examples.
\end{abstract}

\section{Introduction}
\label{introduction}
Equilibrium problems unify many important problems, such as optimization problems, variational inequality problems and fixed point
problems, saddle point (minimax) problems, Nash equilibria
problems and complementarity problems. As far as we know, the term "equilibrium problem" was
coined in 1992 by Muu and Oettli \cite{oettli} and has been elaborated further by Blum and Oettli \cite{blum}. The equilibrium problem (shortly, EP) is also known as the Ky Fan inequality since Fan \cite{fan} gave the first existence result of solutions of the EP. Thanks to its wide applications, many results concerning the existence
of solutions for equilibrium problems have been established and generalized by a number of authors (e.g., see \cite{kassay,mosco,jafari,yuan} and the references therein). One of the most interesting and important problems in the equilibrium problem
theory is the study of efficient iterative algorithms for finding approximate
solutions, and the convergence analysis of algorithms. Serveral methods have been
proposed to solve equilibrium problems in finite and infinite dimensional
spaces (see, e.g., \cite{CH,DMQ,Iusem1,khv,quoc1,quoc,santos1,TTa,TT,vinh} and the references theirein). In \cite{CH,TTa,TT} the authors introduced general iterative schemes based on
the proximal method, the viscosity approximation method and the hybrid method for finding a common element of the set of fixed points of a nonexpansive
mapping and the set of solutions of the equilibrium problem. %Also, Takahashi and Takahashi \cite{TTa} (2007) combined the auxiliary problem principle and viscosity method to 
But in
the proximal method we must solve an regularized equilibrium problem at each iteration of the method. This task is not easy. To overcome this difficulty, Antipin \cite{antipin} and Quoc et al. \cite{quoc} replaced the regularized equilibrium problem by two strongly convex optimizations, which seem computationally easier than solving the regularized equilibrium problem in
the proximal method. Their method is known under the name of the extragradient method. The reason is that when the problem (EP) is a variational inequality problem, this method reduces to the classical extragradient method introduced by Korpelevich \cite{kor}. In 2008, Quoc et al. \cite{quoc} extended the extragradient algorithm for Bregman distance case and proved some important results as the foundation for later studies. It was proved that if the bifunction associated with the (EP) is pseudomonotone and satisfies a Lipschitz-type condition
then the extragradient method is weakly convergent in the framework of Hilbert spaces. Since then, many variants of the extragradient algorithm were developed to improve the efficiency of the method, see \cite{hieu1,quoc1,strodiot2,Vu,strodiot3} for a survey.
%$$L(x,y):=G(y)-G(x)-\langle \Delta x,y-x\rangle,$$
%where $G: \mathbb{R}^n\to\mathbb{R}$ is a strongly convex (with modulus $\beta>0$) and continuously differentiable function. 
In most algorithms, at each iteration, it must either solve two strongly convex programming problems or solve one strongly convex programming problem with one additional projection onto the feasible set. There is even an algorithm that solve three strongly convex programming problems at each iteration. Therefore, the evaluation of the subprogram involved in such algorithms is in general very expensive if the bifunctions and the feasible sets have complicated structures. For more details, see for instance \cite{strodiot1,Lyashko,strodiot2,strodiot3}.

% However, the Lipschitz-type condition is in general difficult to check in practice. Furthermore, solving the strongly convex optimizations is expensive excepts special cases when the feasible set has a simple structure. 

%We already know that one of the most effective numerical techniques is the projection type method. Korpelevich \cite{korpelevich}, for instance, proposed an extragradient method for monotone variational inequality in finite dimensional spaces. For some related works, we refer to \cite{censor1,solodov,malitsky1} and the references therein. Especially, Malitsky and Semenov

Note that the extragradient algorithm must solve two strongly convex programming problems at each iteration. Therefore, their computations are expensive if the bifunctions and the feasible sets have complicated structures. These observations lead us to the following question.\vskip 0.2cm
%To overcome these obstacles, very recently, Hieu \cite{hieu1} extended the subgradient extragradient algorithm of Censor et al. \cite{censor1} to the equilibrium problem (\ref{pt1}) with $f: H\times H\to\mathbb{R}$, where $H$ is a real Hilbert space.  In Hieu's algorithm, the subprogram over the feasible set $C$ in the second step of (\ref{extra}) was replaced by the one over a halfspace. Therefore, compared with Hieu's algorithm, the algorithm (\ref{extra}) is, in general, computationally more expensive at each iteration. Indeed, the increased computational expense relates to solving a subprogram over the feasible set $C$ (many constraints) in the second step of (\ref{extra}), which is clearly more difficult than solving a subprogram over a halfspace (one constraint), as in the second step of Hieu's algorithm. However, the drawback of both (\ref{extra}) and Hieu's algorithm is that $f$ is evaluated twice in the first argument at each iteration. This also may seriously affect the efficiency of the method.
\vskip 0.2cm
{\bf Question.} {\it Can we improve the extragradient algorithm such that we use only one strongly convex programming problem at each iteration?}

\vskip 0.2cm
In this paper, we give a positive answer to this question. Motivated and inspired by the algorithms in \cite{antipin,malitsky,quoc,santos1}, we will introduce some new algorithms for
solving the EP. The advantage of our methods is that it only requires solving one strongly convex optimization problem or computing one projection onto the feasible set. Besides, the assumptions on $f$ can be relaxed and the convergence is still guaranteed.
Numerical examples are presented to describe the efficiency of the proposed approach.
%Theoretical analysis and experimental results show that our algorithm is more efficient than the previous ones for equilibrium problems. %In this paper, we incorporate the new algorithm into the hybrid projection method in order to obtain a strongly convergent sequence to a solution of the problem.

The rest of the paper is organized as follows. After collecting some definitions and basic results in Section 2, we prove in Section 3 the weak convergence of the proposed algorithm. 
%In Section 4, we apply the proposed algorithm to variational inequalities.
In Section 4, we deal with strong convergence by using strong pseudomonotonicity. The particular case
when the equilibrium problem reduces to the variational inequality problem is given in Section 5.
Finally, in Section 6
we provide some numerical results to illustrate the convergence of our
algorithm and compare it with the previous algorithms.
%Our algorithm combine positive features of the Moreau's proximity operator and Popov's algorithm.
%Our results improve some known results of Ceng et al. \cite{ceng1,ceng2}.% We also give some examples to illustrate our results.
\section{Preliminaries}
From now on, we will assume that $C$ is a nonempty closed convex subset of a real Hilbert space $H$ and $f: H\times H\to\mathbb{R}\cup\{+\infty\}$ a bifunction such that $C\times C$ is contained in the domain of $f$. Consider the following problem which is known as an equilibrium problem (see Muu and Oettli \cite{oettli} and Blum and Oettli \cite{blum}):
\begin{equation}\label{pt1}
\text{Find $\bar{x}\in C$ such that}\ f(\bar{x},y)\geq 0\ \ \forall y\in C.
\end{equation}

The set of solutions of the EP (\ref{pt1}) will be denoted by ${\rm Sol}(C,f)$, i.e., $${\rm Sol}(C,f):=\{x\in C: f(x,y)\geq 0\ \ \forall y\in C\}.$$

In 2015, Dong et al. \cite{strodiot1} introduced and analyzed the following General Extragradient Algorithm (EGA) for solving the equilibrium problem (\ref{pt1}):
\begin{equation}\label{eq2}
\begin{cases}
x^0\in C,\\
\bar{x}^k =\underset{y\in C}{\textup{argmin}}\bigg\{\alpha_kf(x^k,y)+\frac{1}{2} \| y-x^k\|^2\bigg\},\\
 \tilde{x}^k = \underset{y\in C}{\textup{argmin}}\bigg\{\beta_kf(\bar{x}^k,y)+\frac{1}{2} \| y-\bar{x}^k\|^2\bigg\},\\
 x^{k+1} = \underset{y\in C}{\textup{argmin}}\bigg\{\beta_kf(\tilde{x}^k,y)+\frac{1}{2} \|y-\tilde{x}^k\|^2\bigg\},
\end{cases}
\end{equation}
where $\alpha_k\geq 0$ and $\beta_k>0$.

It is easy to see that when $\alpha_k=0$ for all $k$, Algorithm GEA reduces to the classical extragradient algorithm \cite{antipin,quoc}. 

In 2017, Hieu \cite{hieu1} introduced an extragradient algorithm for a class of strongly pseudomonotone equilibrium problems as follows.
%The extragradient algorithm for the inequality (\ref{pt1}) is due to Antipin \cite{antipin} (the iterative process (3.7) in the framework of finite-dimensional spaces):
\begin{equation}\label{eq3}
\begin{cases}
x^0\in C,\\
y^{n}=\underset{y\in C}{\textup{argmin}}\bigg\{\lambda_n f(x^n,y)+\dfrac{1}{2}\|y-x^n\|^2\bigg\}, \\
x^{n+1}=\underset{y\in C}{\textup{argmin}}\bigg\{\lambda_n f(y^n,y)+\dfrac{1}{2}\|y-x^{n}\|^2\bigg\},
\end{cases}
\end{equation}
where $\{\lambda_n\}$ is a non-summable and diminishing sequence, i.e.,
\begin{align}\label{eq4}
\lim_{k\to\infty}\lambda_k=0,\ \ \sum_{k=0}^{\infty}\lambda_k=+\infty.
\end{align}

In 2018, Hieu \cite{hieu3} proposed a Popov type algorithm for strongly pseudomonotone equilibrium problems below.
\begin{equation}\label{eq5}
\begin{cases}
x^0,y^0\in C,\\
x^{n+1}=\underset{y\in C}{\textup{argmin}}\bigg\{\lambda_n f(y^n,y)+\dfrac{1}{2}\|y-x^n\|^2\bigg\},\\
y^{n+1}=\underset{y\in C}{\textup{argmin}}\bigg\{\lambda_n f(y^n,y)+\dfrac{1}{2}\|y-x^{n+1}\|^2\bigg\},
\end{cases}
\end{equation}
where $\{\lambda_n\}$ is a nonincreasing sequence satisfying the condition (\ref{eq4}).

Targeting an improvement of the above algorithms, we will introduce the so-called golden ratio algorithm for equilibrium problems in Section \ref{sec:3}.

Now let us start with some concepts and auxiliary results needed in the sequel. Let $H$ be a real Hilbert space endowed with the inner product $\langle .,.\rangle$ and the associated norm $\|.\|$. 
%In this section, we present some preliminary results that we will use in our upcoming results. From now on, we  assume that 
It is easy to see that
\begin{align}\label{hbh}
\|tx+(1-t)y\|^2=t\|x\|^2+(1-t)\|y\|^2-t(1-t)\|x-y\|^2,
\end{align}
for all $x,y\in H$ and for all $t\in\mathbb{R}$.

When $\{x^k\}$ is a sequence in $H$,
we denote strong convergence of $\{x^k\}$ to $x\in H$ by $x^k\to x$ and weak convergence by $x^k\rightharpoonup x$.
Let $C$ be a nonempty closed convex subset of $H$. For every element $x\in H$, there exists a unique nearest point in $C$, denoted by $P_Cx$, that is
$$ ||x-P_Cx||=\min\{||x-y||:\ y\in C\}. $$

The operator $P_C$ is called the \textit{metric projection} of $H$ onto $C$ and some of its properties are summarized in the next lemma, see e.g., \cite{GR84}.
\begin{lemma}\label{hchieu}
	Let $C\subseteq H$ be a closed convex set, $P_C$ fulfils the following:
	\begin{enumerate}
		\item[\textup{(1)}] $\langle x-P_Cx, y-P_Cx\rangle\leq 0$ for all $x\in H$ and $y\in C$;\label{hchieu1}
		\item[\textup{(2)}] $\|P_Cx-y\|^2\leq \|x-y\|^2-\|x-P_Cx\|^2$ for all $x\in H$, $y\in C$.\label{hchieu2}
	\end{enumerate}
\end{lemma}

For a proper, convex and lower semicontinuous function $g:H\to(-\infty,\infty]$ and $\gamma>0$, the Moreau envelope of $g$ of parameter $\gamma$ is the convex function

$$ \leftidx{^\gamma}g(x)=\inf_{y\in H}\Big\{g(y)+\dfrac{1}{2\gamma}\|y-x\|^2\Big\}\ \ \forall x\in H. $$

For all $x\in H$, the function $$y\mapsto g(y)+\dfrac{1}{2\gamma}\|y-x\|^2$$ is proper,
strongly convex and lower semicontinuous, thus the infimum is attained, i.e., $\leftidx{^\gamma}g: H\to\mathbb{R}$. 

The unique minimum of
\begin{equation}\label{mor}
y\mapsto g(y)+\dfrac{1}{2}\|y-x\|^2
\end{equation}
is called proximal point of $g$ at $x$ and it is denoted by $\prox_g(x)$. The operator
\begin{align*}
\prox_g(x): H&\to H\\
x&\mapsto \underset{y\in H}{\textup{argmin}}\bigg\{g(y)+\dfrac{1}{2\gamma}\|y-x\|^2\bigg\}
\end{align*}
is well-defined and is said to be the proximity operator of $g$. When $g=\iota_C$ (the indicator function of the convex set $C$), one has 
$$\prox_{\iota_C}(x)=P_C(x)$$
for all $x\in H$.

We also recall that the subdifferential of $g: H\to(-\infty,\infty]$ at $x\in H$ is defined as the set of all subgradient of $g$ at $x$:
$$ \partial g(x):=\{w\in H: g(y)-g(x)\geq \langle w,y-x\rangle\ \ \forall y\in H\}.$$

%Therefore, $0\in \partial g(x)\Longleftrightarrow g(x)=\underset{y\in H}{\min}g(y)$.
The normal cone of $C$ at $x\in C$ is defined by
$$ N_C(x):=\{q\in H: \langle q,y-x\rangle\leq 0 \ \forall y\in C\}. $$

We now recall classical concepts of monotonicity for nonlinear operators.
%We need the following definitions taken from \cite{bianchi,blum,kamada}.
\begin{definition}\label{mono1}\upshape (see \cite{kamada}) An operator $A: C\to H$ is said to be
	\begin{enumerate}[label=(\alph*),leftmargin=1.2\parindent]
		\item [\upshape (1)] monotone on $C$  if
		$$ \langle Ax-Ay,x-y\rangle\geq 0\ \ \forall x,y\in C. $$
		\item [\upshape (2)] pseudomonotone on $C$ if
		$$ \langle Ax,y-x\rangle\geq 0\Longrightarrow \langle Ay,x-y\rangle\leq 0\ \ \forall x,y\in C. $$
		\item [\upshape (3)] strongly pseudomonotone on $C$ with modulus $\gamma>0$ if there exists $\gamma>0$ such that for any $x,y\in C$
		\begin{equation}
		\langle Ax,y-x\rangle\geq 0\Longrightarrow \langle Ay,x-y\rangle\leq -\gamma\|x-y\|^2\ \ \forall x,y \in C.
		\end{equation}
	\end{enumerate}
\end{definition}

Analogous to Definition \ref{mono1}, we have the following concepts for equilibrium problems.
\begin{definition}\label{mono2}\upshape (see \cite{dinh}) The bifunction $f: H\times H\to \mathbb{R} \cup\{+ \infty\}$ is said to be
	\begin{enumerate}[label=(\alph*),leftmargin=1.2\parindent]
		\item [\upshape (1)] monotone on $C$ if
		$$ f(x,y)+f(y,x)\leq 0\ \ \forall x,y\in C. $$
		\item [\upshape (2)] pseudomonotone on $C$ if
		$$ f(x,y)\geq 0\Longrightarrow f(y,x)\leq 0\ \ \forall x,y\in C. $$
		\item[\textup{(3)}] strongly pseudomonotone on $C$ with modulus $\gamma>0$ if there exists $\gamma>0$ such that for any $x,y\in C$
		$$ f(x,y)\geq 0\Longrightarrow f(y,x)\leq -\gamma \|x-y\|^2. $$
	\end{enumerate}
\end{definition}
\begin{remark}\upshape
	It is obvious that if $A: C\to H$ is monotone (pseudomonotone) on $C$ in the sense of Definition \ref{mono1} then the corresponding bifunction defined by $f(x,y)= \langle Ax,y-x\rangle$ is monotone (pseudomonotone) on $C$ in the sense of Definition \ref{mono2}.
\end{remark}
\begin{example}\upshape
	Suppose that $H=L^2([0,1])$ with the inner product
	$$\langle x,y\rangle:=\int_{0}^{1}x(t)y(t)dt,\ \forall x,y\in H $$
	and the induced norm
	$$\| x\|:=\bigg(\int_{0}^{1}|x(t)|^2dt\bigg)^{\frac{1}{2}},\ \forall x\in H.$$
	
	%Let $\alpha$ and $\beta$ be two positive real numbers such that $\alpha>\beta>\frac{\alpha}{2}$. 
	Let us set
	$$ C=\{x\in H: \|x\|\leq 1\},\ \ f(x,y)=\bigg\langle \dfrac{x}{1+\|x\|^2},y-x\bigg\rangle.$$
	
	We now show that $f$ is strongly pseudomonotone on $C$. Indeed, let $x,y\in C$ be such that $f(x,y)=\Big\langle \dfrac{x}{1+\|x\|^2},y-x\Big\rangle\geq 0$. This implies that $\langle x, y-x\rangle\geq 0$. Consequently,
	\begin{align*}
	f(y,x)&=\bigg\langle \dfrac{y}{1+\|y\|^2},x-y\bigg\rangle\\
	&\leq\dfrac{1}{1+\|y\|^2}(\langle y, x-y\rangle-\langle x, x-y\rangle)\\
	&\leq-\dfrac{1}{2}\|x-y\|^2\\
	&=-\gamma\|x-y\|^2,
	\end{align*}
	where $\gamma:=\frac{1}{2}>0$.
	
	On the other hand, $f$ are neither strongly monotone nor monotone on $C$. To see this, we take $x=\sqrt{3t}$, $y=\sqrt{2t}$ and see that
	\begin{align*}
	f(x,y)+f(y,x)&=\Big\langle \dfrac{y}{1+\|y\|^2}-\dfrac{x}{1+\|x\|^2},x-y\Big\rangle\\
	&=\bigg\langle \dfrac{\sqrt{2t}}{2}-\dfrac{2\sqrt{3t}}{5},\sqrt{3t}-\sqrt{2t}\bigg\rangle\\
	&=\dfrac{1}{2}\bigg(\dfrac{\sqrt{2}}{2}-\dfrac{2\sqrt{3}}{5}\bigg)(\sqrt{3}-\sqrt{2})>0.
	\end{align*}
	
\end{example}

Before concluding this section, we recall the following lemmas which will be useful for proving the convergence results of this paper.
%\begin{lemma}\label{opi} \textup{(Opial \cite{opial})} Let $\{x^n\}$ be a sequence of elements of the Hilbert space $H$ which converges weakly to $x\in H$. Then we have
%$$\underset{n\to\infty}{\lim\inf}\|x^n-x\|<\underset{n\to\infty}{\lim\inf}\|x^n-y\|\ \ \forall y\in H\setminus\{x\}.$$
%\end{lemma}

\begin{lemma}\label{lem1}\textup{(\cite{Io1})}
	Let $g:H\to\mathbb{R}\cup\{+\infty\}$ be proper, lower semicontinuous, and convex and let $C$ be a nonempty closed and convex subset of $H$. Assume either that $f$ is continuous at some point of $C$, or that there is an interior point of $C$ where $f$ is finite. 
 Then, $x^*\in C$
	is a solution of the convex optimization problem
	$$\min\{g(x) : x\in C\}$$ if and only if $$0\in\partial g(x^*)+N_C(x^*).$$
\end{lemma}
\begin{lemma} {\upshape (See \cite{Ta1})}\label{tan-xu}
	Assume that $\{a_k\}$ and $\{b_k\}$ are two sequences of non-negative numbers such that $$a_{k+1} \leq a_k+b_k \ \forall k \in \mathbb{N}.$$
	 If $\sum_{k=1}^{\infty}b_k<\infty$ then $\lim_{k\to \infty}a_k$ exists.
\end{lemma}

\begin{lemma}\label{opi} \textup{(Opial \cite{opial})}
	Let $H$ be a real Hilbert space and $\{x^k\}$ a sequence in $H$ such that there exists a nonempty closed set $S\subset H$ satisfying
	\begin{enumerate}[label=(\alph*),leftmargin=1.2\parindent]
		\item [\upshape (1)] For every $z\in S$, $\lim\limits_{k\to\infty}\|x^k-z\|$ exists;
		\item [\upshape (2)] Any weak cluster point of $\{x^k\}$ belongs to $S$.
	\end{enumerate}
	Then, there exists $\bar{x}\in S$ such that $\{x^k\}$ converges weakly to $\bar{x}$.
\end{lemma}

\section{Golden ratio algorithm for equilibrium problems}\label{sec:3}
\subsection{The algorithm}
%Let $C$ be a nonempty closed convex subset of $H$ and $f: H\times H \to \mathbb{R} \cup\{+\infty\}$ a bifunction such that $C\times C$ is contained in the domain of $f$. 
In what follows, the following usual conditions will be used:
	\begin{enumerate}[label=(\alph*),leftmargin=1.65\parindent]
		\item[\textup{(A1)}] $f(x,x)=0$ for all $x\in C$;
		\item[\textup{(A2)}] $f$ is pseudomonotone on $C$;
		\item[\textup{(A3)}] For any arbitrary sequence $\{z^k\}$ such that $z^{k}\rightharpoonup z$, if $\underset{k\to\infty}\limsup f(z^{k},y)\geq 0$ for all $y\in C$ then $z\in {\rm Sol}(C,f)$;
		%$f(\cdot,y)$ is weakly upper semicontinuous on $H$;
		\item[\textup{(A4)}]$f(x,\cdot)$ is lower semicontinuous convex and subdifferentiable on $C$ for every $x\in C$;
		\item[\textup{(A5)}] There exist positive numbers $c_1$ and $c_2$ such that the Lipschitz-type condition
		\begin{equation} \label{Lip1}
		f(x,y)+f(y,z)\geq f(x,z)-c_1\|x-y\|^2-c_2\|y-z\|^2
		\end{equation}
		holds for all $x,y,z \in C$;
		\item[\textup{(A6)}]Either ${\rm int} C \not=\emptyset$ or $f(x,.)$ is continuous at some point of $C$ for every $x\in C$;
		\item[\textup{(A7)}] For all bounded sequences $\{x^k\}, \{y^k\}\subset C$ such that $\|x^k-y^{k}\|\to 0$, the inequality
		
		\[
		\limsup_{k\to\infty}f(x^k,y^{k})\ge 0
		\]
		holds;
		\item[\textup{(A8)}] The solution set ${\rm Sol}(C,f)\not=\emptyset$;		
		\item[\textup{(A9)}] $f$ is strongly pseudomonotone on $C$ with modulus $\gamma>0$.		
			\end{enumerate}
\begin{remark}\upshape If $f(x,y)=\langle Ax,y-x\rangle$, where $A: C\to H$ is Lipschitz continuous with constant $L>0$ then
		$f$ satisfies the inequality (\ref{Lip1}) with constants $c_1=c_2=\frac{L}{2}$. Indeed, for each $x,y,z\in C$, we have
		\begin{align*}
		f(x,y)+f(y,z)-f(x,z)&=\langle Ax,y-x\rangle+\langle Ay,z-y\rangle-\langle Ax,z-x\rangle\\
		&=-\langle Ay-Ax,y-z\rangle\\
		&\geq -\|Ax-Ay\|\big\|y-z\big\|\\
		&\geq -L\|x-y\|\big\|y-z\big\|\\
		&\geq -\dfrac{L}{2}\|x-y\|^2-\dfrac{L}{2}\|y-z\|^2\\
		&=-c_1\|x-y\|^2-c_2\|y-z\|^2.
		\end{align*}
		
		Thus $f$ satisfies the inequality (\ref{Lip1}).
		\end{remark}

\begin{remark}\upshape It is easy to see that if $f(.,y)$ is weakly upper semicontinuous for all $y\in C$ then $f$ satisfies the condition (A3), which was first introduced by Khatibzadeh and Mohebbi in \cite{khatibzadeh}. However, the converse is not true in general. To see this, we consider the following counterexample in \cite{khv,khatibzadeh}.
\end{remark}
%This remark is illustrated by the following counterexample modified from Remark 2.1 of \cite{khatibzadeh}.
\begin{example}\upshape
	Let $H=\ell^2$, $C=\{\xi=(\xi_1,\xi_2,...)\in \ell^2:\xi_i\geq 0\ \ \forall i=1,2,...\}$ and
	$$ f(x,y)=(y_1-x_1)\sum_{i=1}^{\infty}(x_i)^2. $$
	%Take $x^k=(\underbrace{0,...,0}_k,1,\frac{1}{2},\frac{1}{3},...)$,
	
	Take $x^k=(0,...,0,\underset{k}1,0,...)$, we have $x^k\rightharpoonup x=(0,...,0,...)$ and $x\in {\rm Sol}(C,f)$. Obviously, there is a $y\in C$ such that
	$$ \limsup_{k\to\infty}f(x^k,y)>0=f(x,y). $$
	
	Then $f(.,y)$ is not weakly upper semicontinuous. We now show that $f$ satisfies the condition (A3). If $z^k=(z_1^k,z_2^k,...)\rightharpoonup z=(z_1,z_2,...)$ is an arbitrary sequence and $\underset{k\to\infty}\limsup f(z^k,y)\geq 0$ for all $y\in C$, then we have
	$$  \limsup_{k\to\infty}(y_1-z_1^k)\sum_{i=1}^\infty (z_i^k)^2\geq 0.$$
	
	Since $\underset{k\to\infty}\lim (y_1-z_1^k)=y_1-z_1$, we get
	$$  (y_1-z_1)\limsup_{k\to\infty}\sum_{i=1}^\infty (z_i^k)^2\geq 0,$$
	thus $y_1\geq z_1$. Hence, $f(z,y)\geq 0$ for all $y\in C$, i.e., $f$ satisfies the condition (A3).
\end{example}

From the above observations, it is clear that our conditions (A3) and (A7) are weaker than the conditions (A4) and (A6) in \cite{Lyashko}, respectively.

\begin{remark}\label{chuy3.2}
	\upshape The condition (A7) was introduced by Kassay et al. \cite{khv}. Under (A1), we will show that the assumption (A7) is weaker than the one below, which was considered in \cite{hieu1,hieu2,strodiot2,strodiot3} (see also the references therein).
	\begin{enumerate}[label=(\alph*),leftmargin=1.85\parindent]
		\item[(A7')] $f$ is jointly weakly lower semicontinuous on the product $C\times C$.
	\end{enumerate}

Indeed, to prove that (A7') implies (A7), let $\{x^k\}$, $\{y^k\}$ be bounded sequences in $C$ with $\|x^k-y^{k}\|\to 0$. Thus there exists a subsequence $\{x^{k_l}\}$ of $\{x^k\}$ converging weakly to $\bar{x}\in C$. By the assumption, the subsequence $\{y^{k_l}\}$ converges weakly to the same $\bar{x}$. Hence,
	
	\[
	\limsup_{k\to\infty}f(x^k,y^{k})\ge \limsup_{l\to\infty}f(x^{k_l},y^{k_l})\ge \liminf_{l\to\infty}f(x^{k_l},y^{k_l})\ge f(\bar{x}, \bar{x}) = 0.
	\]
	
	%We also note that the bifunction $f$ in the assumption (A4) of \cite{hieu1} need not to be subdifferentiable with respect to the second argument.
	%In \cite{hieu-a}, the author assumed that $f$ is convex and subdifferentiable on $H$ with respect to the second argument (the condition (A4)). On the other hand, from his condition (A3), we have the weak lower semicontinuity (which implies the lower semicontinuity) of $f$ on the second argument. Therefore, the assumption of subdifferentiability is indeed not necessary.
\end{remark}

We are now in a position to  describe a new algorithm for pseudomonotone equilibrium problems.
\begin{algorithm}[H]
	\caption{(Golden ratio algorithm for equilibrium problems)}\label{alg1}
	{\bf Initialization:} Let $\varphi=\frac{\sqrt{5}+1}{2}$ be the golden ratio, i.e., $\varphi^2=\varphi+1$. Choose the parameter $\lambda$ such that
	\begin{align}
	0<\lambda\leq\min\bigg\{\frac{\varphi}{4c_1}, \frac{\varphi}{4c_2}\bigg\}.
	\end{align}
	
	Select initial $x^0\in C$, $y^1\in C$ and set $k:=0$.\vskip 0.1cm\noindent
	{\bf Iterative Step:} Given $x^{k-1}$ and $y^k$ {\upshape(}$k\geq 1${\upshape)}, compute
	\begin{align}
	&x^k=\dfrac{(\varphi-1)y^k+x^{k-1}}{\varphi},\label{e7}\\
	&y^{k+1}={\rm argmin}\Big\{\lambda f(y^k,y)+\dfrac{1}{2}\|y-x^k\|^2: y\in C\Big\}.\label{e8}
	\end{align}
	\vskip 0.1cm\noindent
	{\bf Stopping Criterion:} If $y^{k+1}=y^k=x^k$ then stop. Otherwise, let $k:=k+1$ and return to {\bf Iterative Step}.
\end{algorithm}
\begin{remark}\upshape
	For comparison with algorithms (\ref{eq2}), (\ref{eq3}) and (\ref{eq5}), our Algorithm \ref{alg1} requires, at each iteration, only one strongly convex optimization problem.
\end{remark}

\subsection{Convergence analysis}
We first wish to validate the stoping criterion of Algorithm \ref{alg1}.
\begin{lemma}
	Under the conditions (A1), (A4) and (A6), if $y^{k+1}=y^k=x^k$ then $y^k\in {\rm Sol}(C,f)$.
\end{lemma}
\begin{proof}
	If $y^{k+1}=y^k=x^k$ then
	$$ y^{k}=\underset{y\in C}{\textup{argmin}}\bigg\{\lambda f(y^k,y)+\dfrac{1}{2}\|y-y^k\|^2\bigg\}.$$
	
	Therefore, from Lemma \ref{lem1} we have $$0\in\partial\bigg[\lambda f(y^k,.)+\frac{1}{2}\|.-y^k\|\bigg](y^k)+N_C(y^k),$$
	i.e., $0\in\partial(\lambda f(y^k,.))(y^k)+N_C(y^k)$,
	which implies that
	$$ \langle u^k,x-y^k\rangle\geq 0\ \ \forall x\in C, $$
	where $u^k\in \partial(f(y^k,.))(y^k)$. By the assumption (A1), we get
	$$ f(y^k,x)=f(y^k,x)-f(y^k,y^k)\geq \langle u^k,x-y^k\rangle\geq 0\ \ \forall x\in C. $$
	
	This means that $y^k\in {\rm Sol}(C,f)$.
\end{proof}

The next statement plays a crucial role in the proof  the convergence result.
\begin{lemma} \label{lem3.2}
	Let $\{x^k\}$ and $\{y^k\}$ be the sequences generated by Algorithm \ref{alg1} and $z\in {\rm Sol}(C,f)$. Under the conditions (A4), (A5) and (A6), the following inequality holds.
	\begin{align*}
	(1+\varphi)\|x^{k+1}-z\|^2+\frac{\varphi}{2}\|y^k-y^{k+1}\|^2&\leq (1+\varphi)\|x^{k}-z\|^2+\frac{\varphi}{2}\|y^{k-1}-y^k\|^2\notag\\
	&\hskip2.5cm-\varphi\|x^k-y^{k}\|^2+2\lambda f(y^k,z).
	\end{align*}
\end{lemma}
\begin{proof}
	From $y^{k+1}=\underset{y\in C}{\textup{argmin}}\bigg\{\lambda f(y^k,y)+\dfrac{1}{2}\|y-x^k\|^2\bigg\}$ and Lemma \ref{lem1}, we have 
	$$ 0=\lambda g^k+y^{k+1}-x^k+q, $$
	where $g^k\in \partial f(y^k,.)(y^{k+1})$ and $q\in N_C(y^{k+1})$. Since 
	$$  N_C(y^{k+1})=\{q\in H: \langle q, y-y^{k+1}\rangle\leq 0\ \forall y\in C\},$$
	we have
	$$ \langle y^{k+1}-x^k+\lambda g^k,z-y^{k+1}\rangle\geq 0. $$
	
	Consequently,
\begin{align}\label{11} \langle x^k-y^{k+1},z-y^{k+1}\rangle\leq \lambda \langle g^k,z-y^{k+1}\rangle\leq \lambda (f(y^k,z)-f(y^k,y^{k+1}))
\end{align}
and
\begin{align}\label{12}
\langle x^{k-1}-y^{k},y^{k+1}-y^{k}\rangle\leq \lambda (f(y^{k-1},y^{k+1})-f(y^{k-1},y^{k})).
\end{align}

Combining (\ref{12}) with the fact that 
\begin{align}\label{14-1}
y^k-x^{k-1}=\frac{1+\varphi}{\varphi}(y^k-x^k)=\varphi(y^k-x^k)
\end{align}
 we obtain
\begin{align}\label{14}
\langle \varphi(x^k-y^k),y^{k+1}-y^{k}\rangle\leq \lambda (f(y^{k-1},y^{k+1})-f(y^{k-1},y^{k})).
\end{align}

Summing up (\ref{12}) and (\ref{14}) we get
\begin{align}\label{15}
\langle x^k-y^{k+1},z-y^{k+1}\rangle+\langle \varphi(x^k-y^k),y^{k+1}-y^{k}\rangle&\leq \lambda \big[f(y^{k-1},y^{k+1})-f(y^{k-1},y^{k})\notag\\
&\hskip2cm-f(y^k,y^{k+1})\big]+\lambda f(y^k,z).
\end{align}

Using the identity 
$$\langle a,b\rangle=\frac{1}{2}\big[\|a\|^2+\|b\|^2-\|a-b\|^2\big]$$
we have from (\ref{15}) that
	\begin{align}\label{16}
	\|y^{k+1}-z\|^2&\leq \|x^{k}-z\|^2-\|x^k-y^{k+1}\|^2-\varphi\big[\|x^k-y^{k}\|^2+\|y^{k+1}-y^k\|^2-\|y^{k+1}-x^k\|^2\big]\notag\\
	&\hskip2.5cm+2\lambda \big[f(y^{k-1},y^{k+1})-f(y^{k-1},y^{k})-f(y^k,y^{k+1})\big]+2\lambda f(y^k,z).
	\end{align}
	
	By the assumption (A5), we get
	\begin{align}\label{17}
2\lambda[f(y^{k-1},y^{k+1})-f(y^{k-1},y^{k})-f(y^k,y^{k+1})] 
	&\leq 2\lambda[c_1\|y^{k-1}-y^k\|^2+c_2\|y^k-y^{k+1}\|^2]\notag\\
	&\leq \frac{\varphi}{2}\big[\|y^{k-1}-y^k\|^2+\|y^k-y^{k+1}\|^2\big].
\end{align}
	
	On the other hand, it can be easily seen from (\ref{14-1}) that $y^{k+1}=(1+\varphi)x^{k+1}-\varphi x^{k}$. Hence, we have from (\ref{hbh}) that
	\begin{align}\label{e13}	
	\|y^{k+1}-z\|^2&=(1+\varphi)\|x^{k+1}-z\|^2-\varphi\|x^{k}-z\|^2+\varphi(1+\varphi)\|x^{k+1}-x^k\|^2\notag\\
	&=(1+\varphi)\|x^{k+1}-z\|^2-\varphi\|x^{k}-z\|^2+\frac{1}{\varphi}\|y^{k+1}-x^k\|^2.
	\end{align}
	
	It follows from (\ref{16}), (\ref{17}) and (\ref{e13}) that
	\begin{align*}
(1+\varphi)\|x^{k+1}-z\|^2&\leq (1+\varphi)\|x^{k}-z\|^2-\varphi\big[\|x^k-y^{k}\|^2+\|y^{k+1}-y^k\|^2\big]\notag\\
	&\hskip0.9cm+\frac{\varphi}{2}\big[\|y^{k-1}-y^k\|^2+\|y^k-y^{k+1}\|^2\big]+2\lambda f(y^k,z)
	\end{align*}
or equivalently,
\begin{align}\label{19}
(1+\varphi)\|x^{k+1}-z\|^2+\frac{\varphi}{2}\|y^k-y^{k+1}\|^2&\leq (1+\varphi)\|x^{k}-z\|^2+\frac{\varphi}{2}\|y^{k-1}-y^k\|^2\notag\\
&\hskip2.5cm-\varphi\|x^k-y^{k}\|^2+2\lambda f(y^k,z).
\end{align}
	
	The proof is complete.
\end{proof}

At this point, we can prove the following weak convergence theorem.
\begin{theorem}\label{th3.1}
	Let $f: H\times H\to\mathbb{R}\cup\{+\infty\}$ be a bifunction satisfying the assumptions (A1)-(A8). Then the sequence $\{x^k\}$ generated by Algorithm \ref{alg1} converges weakly to an solution of the EP (\ref{pt1}).
\end{theorem}
\begin{proof}
	We split the proof into several steps:
	\vskip 0.1cm\noindent
	{\bf Step 1:} We first show the boundedness of the sequence $\{x^k\}$.
	Let $z\in {\rm Sol}(C,f)$. It follows from the pseudomonotonicity of $f$ that $f(y^k,z)\leq 0$, then the inequality (\ref{19}) of Lemma \ref{lem3.2} implies
	\begin{align}\label{20}
	(1+\varphi)\|x^{k+1}-z\|^2+\frac{\varphi}{2}\|y^k-y^{k+1}\|^2&\leq (1+\varphi)\|x^{k}-z\|^2+\frac{\varphi}{2}\|y^{k-1}-y^k\|^2\notag\\
	&\hskip3.6cm-\varphi\|x^k-y^{k}\|^2.
	\end{align}
	
	From this we infer that the sequence $\big\{(1+\varphi)\|x^{k}-z\|^2+\frac{\varphi}{2}\|y^{k-1}-y^k\|^2\big\}$ is convergent. Therefore, the sequence $\{\|x^k-z\|\}$ is bounded, and so is $\{x^k\}$. %Hence, there exists $\bar{x}\in C$ and a subsequence $\{x^{k_l}\}$ of $\{x^k\}$ such that $x^{k_l}\rightharpoonup \bar{x}$. 
	Moreover,
		\begin{align}\label{22}
\lim_{k\to\infty}\|x^k-y^{k}\|=0
	\end{align}
	and also by (\ref{14-1})
	\begin{align}\label{23}
\lim_{k\to\infty}\|y^{k+1}-x^{k}\|=0.
\end{align}
	
This together with (\ref{22}) implies that
	\begin{align}\label{24}
	\lim_{k\to\infty}\|y^{k+1}-y^{k}\|=0.
	\end{align}
	\vskip 0.1cm\noindent
	{\bf Step 2:} Let us show that any weakly cluster point of the sequence  $\{x^k\}$ belongs to the solution set ${\rm Sol}(C,f)$.
	
Indeed, let $\bar{x}$ be an arbitrary weakly cluster point of $\{x^k\}$.  Since $\{x^k\}$ is bounded, there exists a subsequence $\{x^{k_l}\}$ of $\{x^k\}$ such that $x^{k_l}\rightharpoonup \bar{x}$. 
From (\ref{22}) we have $y^{k_l}\rightharpoonup \bar{x}\in C$.
	
	It follows from $$y^{k+1}=\underset{y\in C}{\textup{argmin}}\bigg\{\lambda f(y^k,y)+\dfrac{1}{2}\|y-x^{k}\|^2\bigg\}$$ and Lemma \ref{lem1} that there exist $w^{k+1}\in \partial f(y^k,.)(y^{k+1})$ and $q^{k+1}\in N_C(y^{k+1})$ such that
	$$ 0=\lambda w^{k+1}+y^{k+1}-x^{k}+q^{k+1}. $$
	
	From the definition of $N_C(y^{k+1})$, we deduce that
	$$ \langle x^{k}-y^{k+1}-\lambda w^{k+1},y-y^{k+1} \rangle \leq 0\ \ \forall y\in C, $$ or
	$$   \langle x^{k}-y^{k+1},y-y^{k+1} \rangle \leq \langle \lambda w^{k+1},y-y^{k+1} \rangle\ \ \forall y\in C. $$
	
	On the other hand, since $w^{k+1}\in \partial f(y^k,.)(y^{k+1})$, we get
	$$ \langle w^{k+1},y-y^{k+1} \rangle \leq f(y^k,y)-f(y^k,y^{k+1})  \ \forall y \in C. $$
	
	Hence, we arrive at
	\begin{equation}\label{pt3-18}
	\frac {\langle x^{k}-y^{k+1},y-y^{k+1}\rangle}{\lambda}\leq f(y^k,y)-f(y^k,y^{k+1})\ \ \forall y \in C.
	\end{equation}
	
	Since the left-hand side converges to zero, replacing $k$ in (\ref{pt3-18}) by $k_l$ we have by (\ref{23}) and the assumption (A7) that
	$$
	0\le \limsup_{l\to\infty}f(y^{k_l}, y^{k_l+1}) =  \limsup_{l\to\infty}\left(\frac{\langle x^{k_l}-y^{k_l+1},y-y^{k_l+1}\rangle}{\lambda} + f(y^{k_l}, y^{k_l+1})\right)$$
	$$
	\le \limsup_{l\to\infty}f(y^{k_l}, y)\ \ \forall y \in C.
	$$
	
	Now under the condition (A3), we obtain, $\bar x \in {\rm Sol}(C,f)$.
		\vskip 0.1cm\noindent
		{\bf Step 3:} We claim that $x^k\rightharpoonup \bar{x}$. Since $\bar{x}$ is an arbitrary weakly cluster point we can conclude that the set of all weakly cluster points belongs to the solution set ${\rm Sol}(C,f)$. Taking into account the convergence of the sequence $\big\{(1+\varphi)\|x^{k}-z\|^2+\frac{\varphi}{2}\|y^{k-1}-y^k\|^2\big\}$ and (\ref{23}), we deduce that the sequence $\big\{\|x^{k}-z\|\big\}$ is convergent.
		Hence, it follows from Lemma \ref{opi} that the  sequence $\{x^k\}$ weakly converges to a solution of the equilibrium problem (\ref{pt1}). This completes the proof.
\end{proof}
\begin{remark}\upshape
Theorem \ref{th3.1} extends, improves, supplements, and develops the results of \cite{antipin,strodiot1,Lyashko,quoc} in the following aspects:
\begin{enumerate}
\item[(1)] In comparison with \cite{antipin,strodiot1,Lyashko,quoc}, Algorithm \ref{alg1} has the advantage that our method consists of one strongly convex programming problem instead of two or three ones as the methods of \cite{strodiot1,Lyashko,quoc}.
\item[(2)] The continuity imposed on $f$ is relaxed.
\item[\textup{(3)}] The sequence $\{x^k\}$ generated by Algorithm \ref{alg1} is not Fej\'{e}r monotone. Therefore, our proof techniques are different from those in \cite{antipin,strodiot1,quoc}.

\end{enumerate}
\end{remark}
\section{Strong convergence of the golden ratio algorithm}
We will use the strong pseudomonotonicity of the
bifunction $f$ to establish the strong convergence of the gold ratio algorithm.
\subsection{An algorithm without knowledge of Lipschitz-type constants}
In general, the Lipschitz-type condition (\ref{Lip1}) is not
satisfied, and even if $f$ satisfies (\ref{Lip1}) then finding the constants $c_1$ and $c_2$ is not easy. To overcome this drawback, we propose the following algorithm.
\begin{algorithm}[H]
	\caption{(Golden ratio algorithm without knowledge of Lipschitz-type constants)}\label{alg2}
	{\bf Initialization:} Let $\varphi=\frac{\sqrt{5}+1}{2}$. Take a positive sequence $\{\lambda_k\}$ satisfying
	\begin{align}\label{e20}
    &\lim_{k\to\infty}\lambda_k=0,\ \ \sum_{k=0}^{\infty}\lambda_k=+\infty.
	\end{align}
	
	Select initial $x^0\in C$, $y^1\in C$ and set $k:=0$.\vskip 0.1cm\noindent
	{\bf Iterative Step:} Given $x^{k-1}$ and $y^k$ {\upshape(}$k\geq 1${\upshape)}, compute
	\begin{align*}
	&x^k=\dfrac{(\varphi-1)y^k+x^{k-1}}{\varphi},\\
	&y^{k+1}={\rm argmin}\Big\{\lambda_k f(y^k,y)+\dfrac{1}{2}\|y-x^k\|^2: y\in C\Big\}.
	\end{align*}
	\vskip 0.1cm\noindent
	{\bf Stopping Criterion:} If $y^{k+1}=y^k=x^k$ then stop. Otherwise, let $k:=k+1$ and return to {\bf Iterative Step}.
	\end{algorithm}

We now state and prove the following strong convergence result for Algorithm \ref{alg2}.

\begin{theorem}\label{th4.1}
	Under the assumptions (A1), (A4)-(A6), (A8) and (A9), the sequence $\{x^k\}$ generated by Algorithm \ref{alg2} converges strongly to the unique solution of the EP (\ref{pt1}).
\end{theorem}
\begin{proof}
Arguing as the proof of the Lemma \ref{lem3.2} we have
\begin{align}\label{pt21}
	\|y^{k+1}-z\|^2&\leq \|x^{k}-z\|^2-\|x^k-y^{k+1}\|^2-\varphi\big[\|x^k-y^{k}\|^2+\|y^{k+1}-y^k\|^2-\|y^{k+1}-x^k\|^2\big]\notag\\
	&\hskip 1.3cm+2\lambda_k \big[f(y^{k-1},y^{k+1})-f(y^{k-1},y^{k})-f(y^k,y^{k+1})\big]+2\lambda_k f(y^k,z).
	\end{align}
	
	By the assumption (A5), we get
	\begin{align}\label{222}
	2\lambda_k[f(y^{k-1},y^{k+1})-f(y^{k-1},y^{k})-f(y^k,y^{k+1})] 
	&\leq 2\lambda_k[c_1\|y^{k-1}-y^k\|^2+c_2\|y^k-y^{k+1}\|^2].
	\end{align}

	It follows from (\ref{e13}), (\ref{pt21}) and (\ref{222}) that
	\begin{align*}
	(1+\varphi)\|x^{k+1}-z\|^2&\leq (1+\varphi)\|x^{k}-z\|^2-\varphi\big[\|x^k-y^{k}\|^2+\|y^{k+1}-y^k\|^2\big]\notag\\
	&\hskip0.9cm+2\lambda_k[c_1\|y^{k-1}-y^k\|^2+c_2\|y^k-y^{k+1}\|^2]+2\lambda_k f(y^k,z).
	\end{align*}
	
	Consequently,
	\begin{align}\label{pt23}
	(1+\varphi)\|x^{k+1}-z\|^2&+\frac{\varphi}{2}\|y^k-y^{k+1}\|^2\leq (1+\varphi)\|x^{k}-z\|^2+\frac{\varphi}{2}\|y^{k-1}-y^k\|^2\notag\\
	&-\Big(\frac{\varphi}{2}-2\lambda_kc_1\Big)\|y^{k-1}-y^k\|^2-\Big(\frac{\varphi}{2}-2\lambda_kc_1\Big)\|y^k-y^{k+1}\|^2\notag\\
	&\hskip5.1cm-\varphi\|x^k-y^{k}\|^2-2\gamma\lambda_k\|y^{k}-z\|^2,
	\end{align}
	where the last inequality is obtained from the strong pseudomonotonicity of $f$.
	
	Since $\lim_{k\to\infty}\lambda_k=0$, there exists $k_0$ such that $\frac{\varphi}{2}-2\lambda_kc_1>\frac{\varphi}{4}$ and $\frac{\varphi}{2}-2\lambda_kc_2>\frac{\varphi}{4}$ for all $k\geq k_0$. This together with (\ref{pt23}) implies that the sequence $\big\{(1+\varphi)\|x^{k}-z\|^2+\frac{\varphi}{2}\|y^{k-1}-y^k\|^2\big\}$ is convergent and
		\begin{align}\label{pt24}
&	\lim_{k\to\infty}\|y^k-y^{k+1}\|=\lim_{k\to\infty}\|x^k-y^{k}\|=0.
	\end{align}

Therefore 				
\begin{align}\label{e25} 
\lim_{k\to\infty}\|x^k-z\|^2\in\mathbb{R}.
\end{align}

	On the other hand, we have from (\ref{pt23}) that
	\begin{align}\label{pt25}
	2\gamma\lambda_k\|y^{k}-z\|^2\leq \sigma_k-\sigma_{k+1}\ \forall k\geq k_0,  	
	\end{align}
	where $$\sigma_k=(1+\varphi)\|x^{k}-z\|^2+\frac{\varphi}{2}\|y^{k-1}-y^k\|^2.$$
	
	We fix a number $N\in \mathbb{N}$ and consider the inequality (\ref{pt25}) for
	all the numbers $k_0,...,N$. Adding these inequalities, we obtain
	\begin{align}\label{pt25}
	2\gamma\sum_{k=k_0}^N\lambda_k\|y^{k}-z\|^2\leq \sigma_{k_0}-\sigma_{N+1}\leq\sigma_{k_0}, 	
	\end{align}
which implies
$$\sum\limits_{k=0}^{\infty}\lambda_k\|y^{k}-z\|^2<+\infty.$$

 Hence, by (\ref{e20}), we have
	\begin{align}\label{e27} 
	\liminf_{k\to\infty}\|y^{k}-z\|=0.
	\end{align}
	
Combining (\ref{pt24}) and (\ref{e27}) we get
	\begin{align}\label{e28} 
\liminf_{k\to\infty}\|x^{k}-z\|=0.
\end{align}

Finally, by (\ref{e25}) and (\ref{e28}) we conclude that $\lim_{k\to\infty}\|x^{k}-z\|=0$. The proof is complete.
\end{proof}

\subsection{An algorithm without Lipschitz-type condition}
To avoid the Lipschitz-type condition (\ref{Lip1}), we introduce the following self-adaptive algorithm.
\begin{algorithm}[H]
	\caption{(Golden ratio algorithm without Lipschitz-type condition)}\label{alg3}
	{\bf Initialization:} Let $\varphi=\frac{\sqrt{5}+1}{2}$. Take a positive sequence $\{\beta_k\}$ satisfying
	\begin{align}\label{pt30}
%	&\lim_{k\to\infty}\lambda_k=0,\ \ \
\sum_{k=0}^{\infty}\beta_k=+\infty,\ \ \sum_{k=0}^{\infty}\beta^2_k<+\infty.
	\end{align}
	
	Select initial $x^0\in C$, $y^1\in C$ and set $k:=0$.\vskip 0.1cm\noindent
	{\bf Iterative Step:} Given $x^{k-1}$ and $y^k$ {\upshape(}$k\geq 1${\upshape)}, compute
	\begin{align}
	&x^k=\dfrac{(\varphi-1)y^k+x^{k-1}}{\varphi}.
	\end{align}
	
	Take $g(y^{k})\in \partial(f(y^k,.))(y^k)$ ($k\geq 1$). Calculate
	\begin{align}\label{e35}
	\eta_k=\max\{1,\|g(y^k)\|\},\ \ \lambda_k=\dfrac{\beta_k}{\eta_k}
	\end{align}
	and
	\begin{align}
	y^{k+1}&=P_C(x^k-\lambda_k g(y^k)).\label{e36}
	\end{align}
	
	\vskip 0.1cm\noindent
	{\bf Stopping Criterion:} If $y^{k+1}=y^k=x^k$ then stop. Otherwise, let $k:=k+1$ and return to {\bf Iterative Step}.
	
\end{algorithm}

The following lemma is quite helpful to analyze the convergence of Algorithm \ref{alg1}.
\begin{lemma}\label{bode3.0}
	If $y^{k+1}=y^k=x^k$ then $y^k\in {\rm Sol}(C,f)$.
\end{lemma}
\begin{proof}
	%\vspace{-1.5em}$ $\newline%
	If $y^{k+1}=y^k=x^k$ then by (\ref{e36}) and Lemma \ref{hchieu} (1), we have 
	\begin{align*}
	\langle y^k-\lambda_k g(y^k)-y^k,y-y^k\rangle&\leq 0\ \ \forall y\in C,
	\end{align*}
	or equivalently,
	\begin{align}
	\langle g(y^k),y-y^k\rangle&\geq 0\ \ \forall y\in C.\label{8}
	\end{align}
	
	Therefore, from (\ref{8}) and by the definition of $\partial(f(y^k,.))(y^k)$, we get
	$$ f(y^k,y)=f(y^k,y)-f(y^k,y^k)\geq \langle g(y^k),y-y^k\rangle\geq 0\ \ \forall y\in C. $$
	
	Hence $y^k\in {\rm Sol}(C,f)$.
\end{proof}

To establish the strong convergence of Algorithm \ref{alg3}, we will use the following requirement:\vskip 0.2cm
\noindent
(A10) If $\{x^k\}\subset C$ is bounded, then the sequence $\{g^k\}$ with $g^k\in \partial(f(x^k,.))(x^k)$ is bounded.
%, where $\partial(f(x^k,.))(x^k)$ stands for the subdifferential of the convex function $f(x,.)$ at $x$.
\vskip 0.2cm
\noindent

We are now in a position to establish the strong convergence of the sequence generated
by Algorithm \ref{alg3}.
\begin{theorem}\label{th4.2}
	Under the assumptions (A1), (A4), (A8), (A9) and (A10), the sequence $\{x^k\}$ generated by Algorithm \ref{alg3} converges strongly to the unique solution of the EP (\ref{pt1}).
\end{theorem}
\begin{proof}
	Write $w^k=x^k-\lambda_kg(y^k)$. Then using Lemma \ref{hchieu2} (2) we know that
	\begin{align}\label{34}
	\|y^{k+1}-z\|^2&\leq\|w^k-z\|^2-\|w^k-y^{k+1}\|^2\notag\\
	&=\|x^k-\lambda_k g(y^k)-z\|^2-\|x^k-\lambda_k g(y^k)-y^{k+1}\|^2\notag\\
	&=\|x^k-z\|^2-\|x^k-y^{k+1}\|^2+2\lambda_k\langle z-y^{k+1}, g(y^k)\rangle\notag\\
	&=\|x^k-z\|^2-\|x^k-y^{k+1}\|^2+2\lambda_k\langle g(y^k), y^k-y^{k+1}\rangle-2\lambda_k\langle g(y^k), y^k-z\rangle.
%	&\leq \|x^k-z\|^2-\|x^k-y^{k+1}\|^2+2\langle \lambda_kA(y^k)-\lambda_{k-1}A(y^{k-1}), y^k-y^{k+1}\rangle\notag\\
%	&\hskip 2.3cm+2\lambda_{k-1}\langle A(y^{k-1}), y^k-y^{k+1}\rangle-2\lambda_k\langle A(z), y^k-z\rangle.
	\end{align}
	
	From (\ref{e36}) and $y^{k+1}\in C$ we have
	\begin{align}\label{35}
	\langle y^{k}-x^{k-1}+\lambda_{k-1}g(y^{k-1}),y^{k+1}-y^{k} \rangle&\geq 0\notag\\
	\Longleftrightarrow 
		\langle \varphi(y^{k}-x^{k})+\lambda_{k-1}g(y^{k-1}),y^{k+1}-y^{k} \rangle&\geq 0.	
	\end{align}
	
		It follows from (\ref{34}) and (\ref{35}) that
\begin{align}\label{36}
\|y^{k+1}-z\|^2
&\leq\|x^k-z\|^2-\|x^k-y^{k+1}\|^2+\langle \varphi(y^{k}-x^{k}),y^{k+1}-y^{k} \rangle\notag\\
&+2\langle \lambda_kg(y^k)-\lambda_{k-1}g(y^{k-1}), y^k-y^{k+1}\rangle-2\lambda_k\langle g(y^k), y^k-z\rangle.
	\end{align}
	
	Moreover,
	\begin{align*}
	\langle \varphi(y^{k}-x^{k}),y^{k+1}-y^{k} \rangle
	&=\varphi\big[\|y^{k+1}-x^k\|^2-\|x^k-y^{k}\|^2-\|y^{k+1}-y^k\|^2\big].
	\end{align*}

This equality, together with (\ref{e13}) and (\ref{36}), yields
\begin{align}\label{38}
(1+\varphi)\|x^{k+1}-z\|^2
&\leq(1+\varphi)\|x^k-z\|^2-\varphi\big[\|x^k-y^{k}\|^2+\|y^{k+1}-y^k\|^2\big]\notag\\
&\hskip1cm+2\langle \lambda_kg(y^k)-\lambda_{k-1}g(y^{k-1}), y^k-y^{k+1}\rangle-2\lambda_k\langle g(y^k), y^k-z\rangle.
\end{align}

Using the definition of the diagonal subdifferential and the fact that $f$ is strongly pseudomonotone on $C$ we have
	\begin{align}\label{39}
	\langle g(y^k), z-y^k\rangle\leq f(y^k,z)\leq -\gamma\|y^k-z\|^2.
\end{align}
	Setting $LS:=\langle \lambda_kg(y^k)-\lambda_{k-1}g(y^{k-1}), y^k-y^{k+1}\rangle$ we find that
	\begin{align}\label{40}
	LS&
	\leq \bigg\|\dfrac{\beta_k}{\eta_k}g(y^k)-\dfrac{\beta_{k-1}}{\eta_{k-1}}g(y^{k-1})\bigg\|\|y^k-y^{k+1}\|\notag\\
	&\leq (\beta_k+\beta_{k-1})\|y^{k+1}-y^k\|\notag\\
	&\leq \dfrac{1}{2}\Big((\beta_k+\beta_{k-1})^2+\|y^{k+1}-y^k\|^2\Big)\notag\\
	&\leq\beta_k^2+\beta_{k-1}^2+\dfrac{1}{2}\|y^{k+1}-y^k\|^2.
	\end{align}
	
In virtue of (\ref{38}), (\ref{38}) and (\ref{40}) we obtain
\begin{align}\label{41}
(1+\varphi)\|x^{k+1}-z\|^2
&\leq(1+\varphi)\|x^k-z\|^2-\varphi\|x^k-y^{k}\|^2-(\varphi-1)\|y^{k+1}-y^k\|^2\big]\notag\\
&\hskip6.3cm-\gamma\lambda_k\|y^k-z\|^2+2\beta_k^2+2\beta_{k-1}^2.
\end{align}

This yields
\begin{align}\label{38}
a_{k+1}
&\leq a_k+b_k,
\end{align}
where $a_k=(1+\varphi)\|x^k-z\|^2$, $b_k=2\beta_k^2+2\beta_{k-1}^2$. 

The use of Lemma \ref{tan-xu} leads to the convergence of the sequence $\{\|x^k-z\|^2\}$, hence $\{x^k\}$ is bounded. From (\ref{pt30}) and (\ref{41}) we get immediately
		\begin{align}
&	\lim_{k\to\infty}\|y^k-y^{k+1}\|=\lim_{k\to\infty}\|x^k-y^{k}\|=0.
\end{align}

Therefore, $\{y^k\}$ is also bounded. Using 
(A10) we infer that there exists $M_1>0$ such that $\|g(y^k)\|\leq M_1$ for all $k\in\mathbb{N}$. Setting $M_2:=\max\{1, M_1\}$ we have from (\ref{e35}) that
$$\beta_k \geq\lambda_k=\dfrac{\beta_k}{\eta_k}\geq \dfrac{1}{M_2}\beta_k\ \ \forall k\in\mathbb{N}, $$
which together with (\ref{pt30}) yields
\begin{align}\label{pt3.4}
\sum_{k=0}^{\infty}\lambda_k=+\infty; \quad \lim_{k \to \infty} \lambda_k = 0.
\end{align}

The rest of the proof is similar to the proof of Theorem \ref{th4.1}, so we omit the details here.

The proof is complete.
\end{proof}
\section{Application to variational inequalities}
If the equilibrium bifunction $f$ is defined by $f(x,y)=\langle Ax,y-x\rangle$ for every $x,y\in C$, with $A:C\to H$, then the equilibrium problem (\ref{pt1}) reduces to the {\it variational
	inequality problem} (VIP):
\begin{equation}\label{vip}
\text{find }x^* \in C \text{ such that } \langle Ax^*,y-x^*\rangle\geq 0\ \ \forall y \in C.%\tag*{$EP(f,C)$}
\end{equation}

The set of solutions of the problem (\ref{vip}) is denoted by ${\rm Sol}(C,A)$.
In this situation, Algorithm \ref{alg1} reduces to the golden ratio algorithm for variational inequalities, which is recently considered by Malitsky \cite{malitsky}.
\begin{algorithm}[H]
	\caption{(Golden ratio algorithm for variational inequalities)}\label{alg4}
	{\bf Initialization:} Let $\varphi=\frac{\sqrt{5}+1}{2}$ and $\lambda>0$.
	
	Select initial $x^0\in C$, $y^1\in C$ and set $k:=0$.\vskip 0.1cm\noindent
	{\bf Iterative Step:} Given $x^{k-1}$ and $y^k$ {\upshape(}$k\geq 1${\upshape)}, compute
	\begin{align}
	&x^k=\dfrac{(\varphi-1)y^k+x^{k-1}}{\varphi}
	\end{align}
	and
	\begin{align}
	y^{k+1}&=P_C(x^k-\lambda Ay^k).
	\end{align}
	
	\vskip 0.1cm\noindent
	{\bf Stopping Criterion:} If $y^{k+1}=y^k=x^k$ then stop. Otherwise, let $k:=k+1$ and return to {\bf Iterative Step}.
	
\end{algorithm}
\begin{remark}\upshape
	Algorithm \ref{alg4} requires, at each iteration, only one projection onto the feasible set $C$.
\end{remark}

%To guarantee that $f$ is jointly weakly lower semicontinuous on $H\times H$, the authors in \cite{strodiot3} required the weak-to-strong continuity of $A: H\to H$, i.e., $A$ is such that for any sequence $\{x^n\}\subset H$,
%\begin{equation}\label{yeu}
%x^n\rightharpoonup x\Longrightarrow Ax^n\to Ax.
%\end{equation}

%As indicated in Section 3, we do not need anymore the joint weak lower semicontinuity of $f$, but only the weak upper semicontinuity of $f(\cdot, y)$. To this end, 
We now remind the following concept for single-valued operators (called $F$-hemicontinuity in \cite{Mau}).
\begin{definition}\label{Fhemi}
	Let $X$ be a normed space with $X^*$ its dual space and $K$ a closed convex subset of $X$. The mapping $A: K\to X^*$ is called $F$-{\it hemicontinuous} iff for all
	$y\in K,$ the function $x\mapsto\langle A(x), x - y \rangle$ is
	weakly lower semicontinuous on $K$ (or equivalently, $x\mapsto\langle A(x), y - x \rangle$ is weakly upper semicontinuous on $K$).
\end{definition}
%Observe that this is exactly the property we need. Although rather implicit, it was widely used within the literature, and (more important) is weaker than weak-to-strong continuity of $A$.
Clearly, any weak-to-strong continuous mapping is also $F$-hemicontinuous, but vice-versa not, as the following example shows.

\begin{example} {\upshape(\cite{gm})} \upshape Consider the Hilbert space $\ell^{2}=\{x=(x^{i})_{i\in
		\mathbf{N}}$ :$\underset{i=1}{\overset{\infty }{\sum }}|x^{i}|^{2}<\infty \}$ and $A$: $\ell^{2}\rightarrow \ell^{2}$ be the
	identity operator. Take an arbitrary sequence $\{x_{n}\}\subseteq
	\ell^{2}$ converging weakly to $\overline{x}.$ Since the function
	$x\longmapsto \|x\|^2$ is continuous and convex, it is
	weakly lower
	semicontinuous. Hence,%
	\begin{equation*}
	\Vert \overline{x}\Vert ^{2}\leq \underset{n\rightarrow \infty }{\lim \inf }%
	\text{ }\Vert x_{n}\Vert ^{2},
	\end{equation*}%
	which clearly implies%
	\begin{equation*}
	\langle \overline{x},\overline{x}-y\rangle \leq \underset{n\rightarrow \infty }{\lim \inf }\text{ }\langle x_{n},x_{n}-y\rangle ,
	\end{equation*}%
	for all $y\in \ell^{2},$ i.e., $A$ is $F$-hemicontinuous.
	
	%On the other hand, we take $x_{n}=e_{n}=(0,0,...,0,1,0,...)$ with $1$ in
	%the $n^{th}$ position. It is obvious that $e_{n}\rightharpoonup 0,$ but $%
	%\{e_{n}\}$ does not have any strongly convergent subsequence, as $\|e_{n}-e_{m}\|=\sqrt{2}$ for $m\not=n$. Therefore, $A$ is not weak-to-strong continuous.
\end{example}
%Taking into account this property, Corollary 4.1 can be weakened by assuming $F$-hemicontinuity of $A$ instead of (4.29). As we have seen, the new assumption (B6) is satisfied by the Lipschitz property of $A$.

%Moreover, the new assumption (B6) is satisfied by the Lipschitz property of $A$.
The following result is an extension of the corresponding result of Malitsky to infinite dimensional spaces.
\begin{corollary}\label{corplus}
	Let $C$ be a nonempty closed convex subset of $H$. Let $A: C\to H$ be a pseudomonotone, $F$-hemicontinuous, Lipschitz continuous mapping with constant $L>0$ such that ${\rm Sol}(C,A)\not=\emptyset$.
	Let $\{x^k\}$, $\{y^k\}$ be the sequences generated by Algorithm \ref{alg4} with $\lambda\in\big(0,\frac{\varphi}{2L}\big]$. Then the sequences $\{x^k\}$ and $\{y^k\}$ converge weakly to the same point $x^*\in {\rm Sol}(C,A)$.
\end{corollary}
\begin{proof}
	For each pair $x,y\in C$, we define
	\begin{equation}\label{VIP}
	f(x,y):= \begin{cases}
	\langle Ax,y-x\rangle,&\text{if} \ x,y\in C,\\
	+ \infty,&\text{otherwise}.
	\end{cases}
	\end{equation}

	From the assumptions, it is easy to check that all the conditions of Theorem \ref{th3.1} are satisfied. Note that the formula (\ref{e8}) of Algorithm \ref{alg1} can be equivalently written as
	\begin{align*}
	y^{k+1}&=\underset{y\in C}{\textup{argmin}}\bigg\{\lambda \langle Ay^k,y-y^k\rangle+\dfrac{1}{2}\|y-x^{k}\|^2\bigg\},\\
	&=\underset{y\in C}{\textup{argmin}}\bigg\{\dfrac{1}{2}\|y-(x^{k}-\lambda Ay^k)\|^2\bigg\}\\
	&=P_C(x^{k}-\lambda Ay^k).
	\end{align*}
	
	By Theorem \ref{th3.1},
	the sequences $\{x^k\}$ and $\{y^k\}$ converge weakly to $x^*\in {\rm Sol}(C,f)$. It means that the sequences $\{x^k\}$ and $\{y^k\}$ converge weakly to $x^*\in {\rm Sol}(C,A)$.
	Hence, the result is true and the proof is complete.
\end{proof}

\section{Preliminary numerical results}

In this section, we provide numerical examples to illustrate our algorithms and compare with other existing algorithms in \cite{strodiot1,hieu1,hieu3}. All
the codes were written in Matlab (R2015a) and run on PC with Intel(R) Core(TM) i3-370M Processor 2.40 GHz.
%In this section, we present some numerical results to test Algorithm \ref{tt3.1}. 
In the numerical results reported in the following tables, \textquoteleft Iter.\textquoteright\ and \textquoteleft Sec.\textquoteright\ denote the number of iterations and
the cpu time in seconds, respectively.
\begin{example}\label{vd1}\upshape
	%We apply Algorithm \ref{alg1} to solve an equilibrium problem arised from a Nash-Cournot equilibrium model recently investigated by Quoc et al.	in \cite{quoc}  . 
Consider the equilibrium problem given in \cite{quoc} where the bifunction 
$$f:\mathbb{R}^5\times \mathbb{R}^5\to \mathbb{R}$$
is defined for every $x,y\in\mathbb{R}^5$ by
$$ f(x,y)=\langle Px+Qy+q,y-x \rangle, $$
where the vector $q\in\mathbb{R}^5$, and the matrices $P$ and $Q$ are two square matrices of order
5 such that $Q$ is symmetric positive semidefinite and $Q-P$ is negative semidefinite.
To illustrate our algorithms, the matrices $P,Q$ and the vector $q$ are chosen as follows:

	$$ q=\begin{bmatrix}
	\ \ \ 1\\
	-2\\
	-1\\
	\ \ \ 2\\
	-1\\
	\end{bmatrix}, P=\begin{bmatrix}
	3.1&2&0&0&0\\
	2&3.6&0&0&0\\
	0&0&3.5&2&0\\
	0&0&2&3.3&0\\
	0&0&0&0&3\\
	\end{bmatrix}, Q=\begin{bmatrix}
	1.6&1&0&0&0\\
	1&1.6&0&0&0\\
	0&0&1.5&1&0\\
	0&0&1&1.5&0\\
	0&0&0&0&2\\
	\end{bmatrix}. $$
	
	The feasible set is
	$$ C=\bigg\{x\in \mathbb{R}^5:\sum_{i=1}^{5}x_i\geq -1,-5\leq x_i\leq 5,\ i=1,\ldots,5\bigg\}. $$
	
%	It is not difficult to see that $f$ satisfies \eqref{Lip1} with constants $c_1=c_2=\frac {\|Q-P\|} 2$. The eigenvalues of the matrix $Q-P$ are $-2.9050,\ -2.7808,\ -1.0000,\ -0.8950,\ -0.7192$. It follows that $Q-P$ is negative definite, and hence, $f$ is monotone.
	
	Then all the conditions (A1)-(A8) of Theorem \ref{th3.1} are satisfied. We will apply Algorithm \ref{alg1} and the algorithm (\ref{eq2}) (GEA) to solve the EP (\ref{pt1}). In both algorithms, we will the same starting point $x^0$, the same step size $\alpha_k=\beta_k=\lambda =0.27$ and the stopping rule $\|y^{k+1}-y^k\|+\|y^k-x^k\|<10^{-6}$ for Algorithm \ref{alg1} and $\|\tilde{x}^k-\bar{x}^k\|<10^{-6}$ for GEA.
	
	In Table \ref{bang2}, we have compared the performance of Algorithm \ref{alg1} (GRA1) with the General Extragradient Algorithm (\ref{eq2}) (Algorithm GEA in \cite{strodiot1}).
	\begin{table} [!ht]
		\centering 
		\begin{tabular}{l c c c c c c}
			\toprule
			& \multicolumn{2}{c}{$x^0=(-1,3,1,1,2)$}& \multicolumn{2}{c}{$x^0=(1,1,1,1,1)$}& \multicolumn{2}{c}{$x^0=(-1,0,0,0,0)$} \\
			\cmidrule(l){2-3} \cmidrule(l){4-5}\cmidrule(l){6-7}
			& Sec. & Iter. & Sec. & Iter. & Sec. &Iter.\\
			\midrule
			GEA & 2.6094 & 40 & 2.6875 & 40 & 2.6094 & 40\\
			GRA1 & 2.4844 & 97 & 2.4688 & 96 & 2.4531 & 96\\
			\bottomrule
		\end{tabular}
		\caption{\small Comparison of Algorithm \ref{alg1} and Algorithm GEA in Example \ref{vd1} with different $x^0$}
		\label{bang2}
	\end{table}
	
	Convergent behavior of two algorithms with different $x^0$ is given in Figures \ref{hinh1}-\ref{hinh3}. In this figure, the value of errors $\|y^{k+1}-y^k\|+\|y^k-x^k\|$ (Algorithm \ref{alg1}) and $\|\tilde{x}^k-\bar{x}^k\|$ (EGM) is represented by the $y$-axis, number of iterations is represented by the $x$-axis.
	\begin{figure}[H]
		\centering
		\includegraphics[scale=0.8]{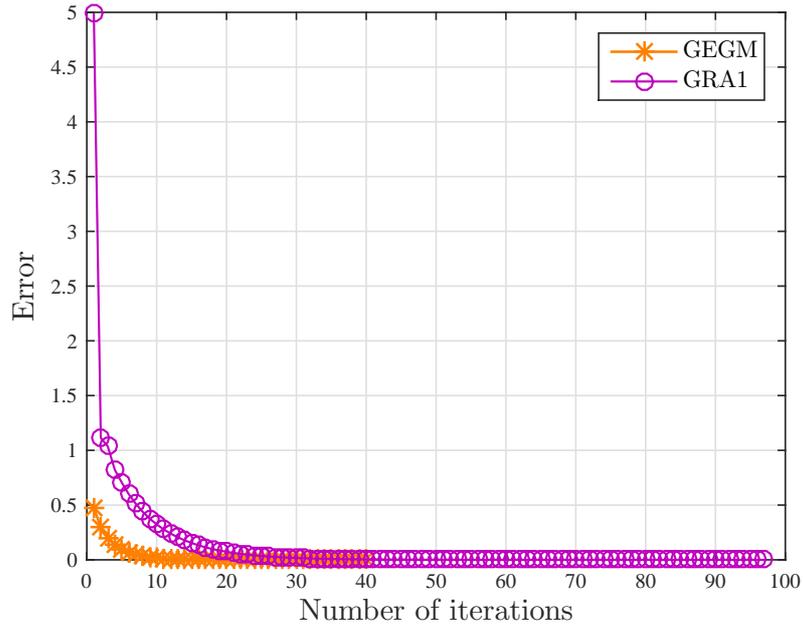}
		\caption{Numerical behavior of two algorithms in Example \ref{vd1} with $x^0=(-1,3,1,1,2)$}\label{hinh1}
	\end{figure}
\begin{figure}[H]
	\centering
	\includegraphics[scale=0.8]{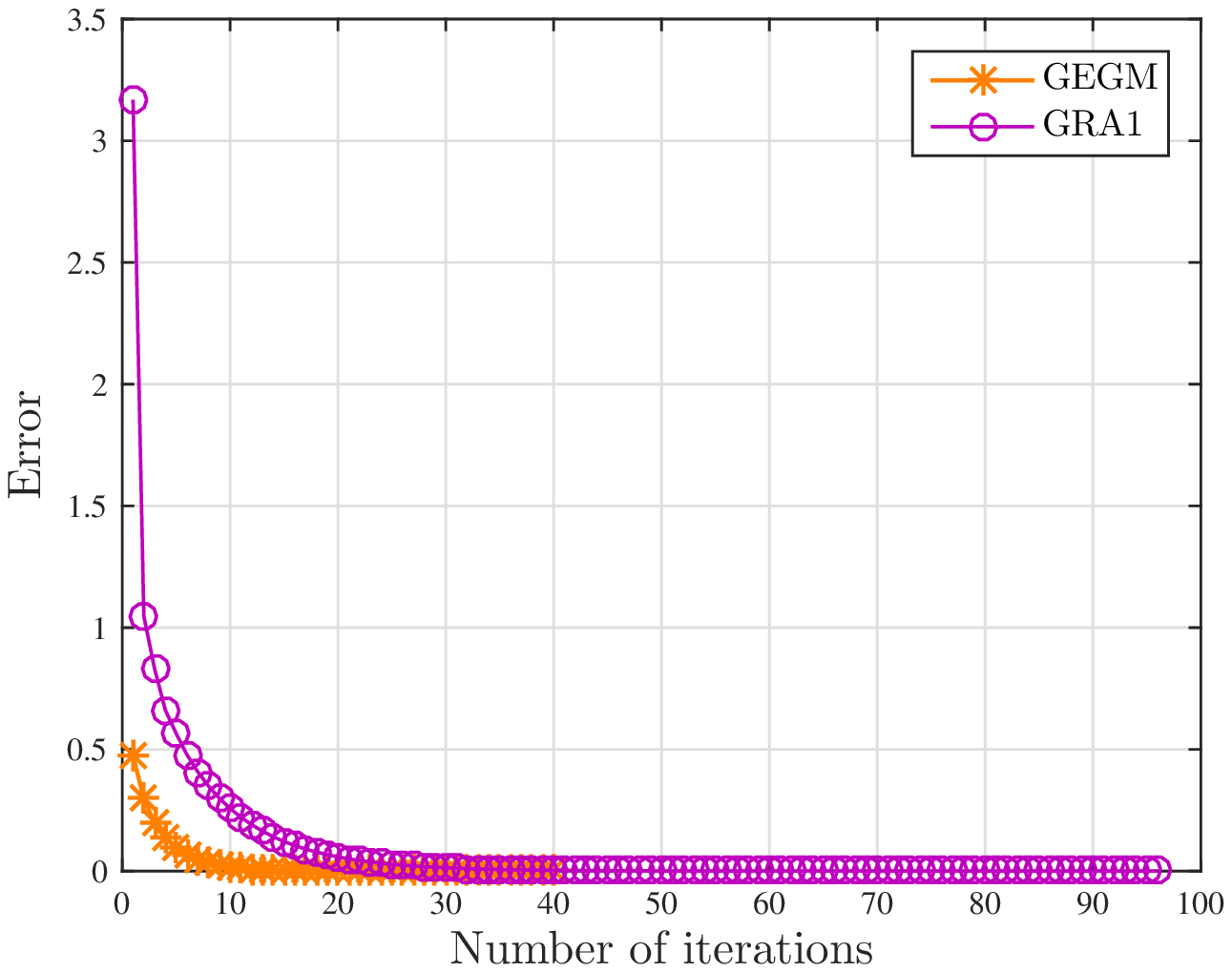}
	\caption{Numerical behavior of two algorithms in Example \ref{vd1} with $x^0=(1,1,1,1,1)$}\label{hinh2}
\end{figure}
\begin{figure}[H]
	\centering
	\includegraphics[scale=0.8]{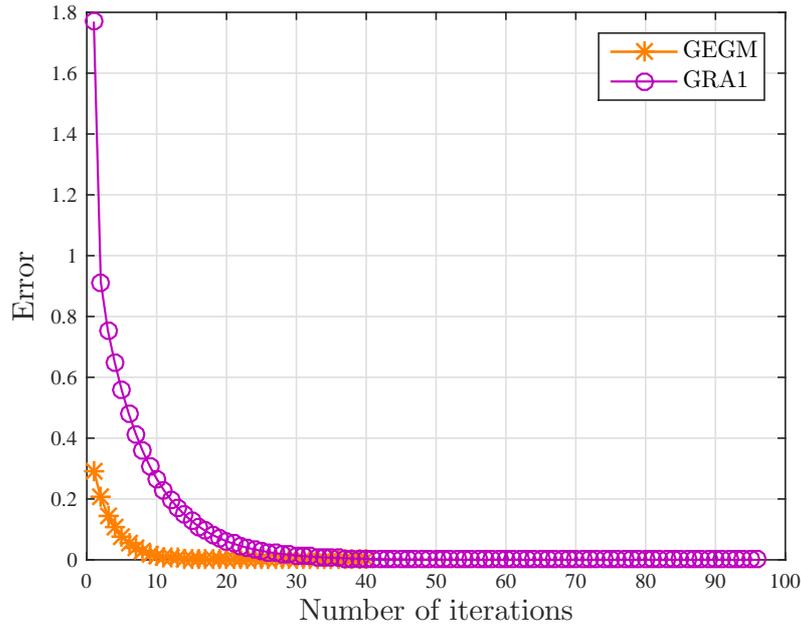}
	\caption{Numerical behavior of two algorithms in Example \ref{vd1} with $x^0=(-1,0,0,0,0)$}\label{hinh3}
\end{figure}
	Let us observe that computational time of Algorithm \ref{alg1} is smaller than that of Algorithm GEA in \cite{strodiot1} but not much for this simple and small example. %However, as mentioned above, the main advantage of Algorithm \ref{alg1} is that it requires only one strongly convex programming problem at each iteration.
\end{example}
\begin{example}\upshape (see also \cite{DMQ})\label{vd2}
	Suppose that $H=L^2([0,1])$ with the inner product
	$$\langle x,y\rangle:=\int_{0}^{1}x(t)y(t)dt,\ \forall x,y\in H $$
	and the induced norm
	$$\| x\|:=\bigg(\int_{0}^{1}|x(t)|^2dt\bigg)^{\frac{1}{2}},\ \forall x\in H.$$
	
	%Let $\alpha$ and $\beta$ be two positive real numbers such that $\beta>\alpha>\frac{\beta}{2}$ and $\beta^2>4\alpha(\beta-\alpha)$. 
	Let us set
	$$ C=\{x\in H: \|x\|\leq 1\},\ \ f(x,y)=\Big\langle \Big(\frac{3}{2}-\|x\|\Big)x,y-x\Big\rangle.$$
	
	We will show that $f$ is strongly pseudomonotone on $C$. Indeed, assume that $x,y\in C$ are such that $f(x,y)=\Big\langle \Big(\frac{3}{2}-\|x\|\Big)x,y-x\Big\rangle\geq 0$. Since $\frac{3}{2}-\|x\|>0$, we have $\langle x,y-x\rangle\geq 0$. Note that
	\begin{align*}
	f(y,x)&=\Big\langle \Big(\frac{3}{2}-\|y\|\Big)y,x-y\Big\rangle\\
	&\leq \Big(\frac{3}{2}-\|y\|\Big)(\langle y,x-y\rangle-\langle x,x-y\rangle)\\
	&=-\Big(\frac{3}{2}-\|y\|\Big)\|x-y\|^2\\
	&\leq -\frac{1}{2}\|x-y\|^2.
	\end{align*}
	
	So $f$ is strongly pseudomonotone on $C$, where $\gamma=\frac{1}{2}$. On the other hand, $f$ are neither strongly monotone nor monotone on $C$. To see this, we take $x=\frac{3}{2}\sqrt{\frac{t}{2}}$, $y=\sqrt{2t}$ and see that
	\begin{align*}
	f(x,y)+f(y,x)&=\bigg\langle \bigg(\frac{3}{2}-\frac{3}{4}\bigg)x,y-x\bigg\rangle+\bigg\langle \bigg(\frac{3}{2}-1\bigg)y,x-y\bigg\rangle\\
	&=\bigg\langle \frac{9}{8}\sqrt{\frac{t}{2}}-\frac{1}{2}\sqrt{2t},\sqrt{2t}-\frac{3}{2}\sqrt{\frac{t}{2}}\bigg\rangle\\
	&=\dfrac{1}{2}\bigg(\frac{9}{8\sqrt{2}}-\frac{\sqrt{2}}{2}\bigg)\bigg(\sqrt{2}-\dfrac{3}{2\sqrt{2}}\bigg)>0.
	\end{align*}
	
	To apply our Algorithm \ref{alg2}, it remains to prove that $f$ satisfies Lipschitz-type condition (A5). Indeed, for all $x,y,z\in C$ we have
	\begin{align}\label{e54}
	f(x,y)+f(y,z)&=f(x,z)+\bigg\langle \bigg(\frac{3}{2}-\|x\|\bigg)x-\bigg(\frac{3}{2}-\|y\|\bigg)y,y-z\bigg\rangle\notag\\
	&=f(x,z)+\frac{3}{2}\langle x-y,y-z\rangle-\langle\|x\|x-y\|y\|,y-z\rangle\notag\\
&=f(x,z)+\frac{3}{4}[\|x-z\|^2-\|x-y\|^2-\|y-z\|^2]\notag\\
&\hskip6cm-\langle\|x\|x-y\|y\|,y-z\rangle.
 	\end{align}
	
	On the other hand, we have
	\begin{align}\label{e55}
\langle\|x\|x-y\|y\|,y-z\rangle&=\langle\|x\|(x-y)+y(\|x\|-\|y\|),y-z\rangle\\
&\leq 2\big\|x-y\|\big|\big\|y-z\big\|\notag\\
&\leq \|x-y\|^2+\|y-z\|^2.
	\end{align}
	
	Combining (\ref{e54}) and (\ref{e55}) we get
	\begin{align*}
	f(x,y)+f(y,z)&\geq f(x,z)-\frac{7}{4}\|x-y\|^2-\frac{7}{4}\|y-z\|^2.
	\end{align*}
	
	Hence, the Lipschitz type inequality (A5) is satisfied. We will apply Algorithm \ref{alg2} (GRA2) to solve the EP (\ref{pt1}) and compare it with Algorithm 1 of \cite{hieu1} (Hieu's algorithm) and Algorithm 3.1 of \cite{hieu3} (Popov's algorithm). To test three algorithms, we take the same parameter $\lambda_k=\frac{40}{k+1}$, $k=0,1,2,...$ and
	\begin{enumerate}[label=(\alph*),leftmargin=1.6\parindent]
		\item[\textup{(i)}] the stopping criterion $\|y^{k+1}-x^k\|+\|y^k-x^k\|<10^{-3}$ for Algorithm \ref{alg2} and Popov's algorithm;  $\|x^k-y^k\|<10^{-3}$ for Hieu's algorithm;
		\item[\textup{(ii)}] the same initial point $x^0$.
	\end{enumerate}

Numerical results of three algorithms are presented in Table \ref{bang1}.
	\begin{table}[H]
		\centering 
		\begin{tabular}{l c c c c}
			\toprule
			& \multicolumn{2}{c}{$x^0=\frac{1}{200}(\sin(-3t)+\cos(-10t)$}& \multicolumn{2}{c}{$x^0=\frac{1}{85}(t^3+1)e^{5t}$}\\
			\cmidrule(l){2-3} \cmidrule(l){4-5}
			& Sec. & Iter. & Sec. & Iter. \\
			\midrule
			Hieu's algorithm &  0.0091149 & 86 & 0.0081864 & 86 \\
			Popov's algorithm & 0.013027 & 118 & 0.012297 & 118 \\
			GRA2&               0.0065136 & 83 & 0.0065654 & 83 \\
			\bottomrule
		\end{tabular}
		\caption{\small Comparison of three algorithms in Example \ref{vd2}}
		\label{bang1}
	\end{table}

\begin{figure}[H]
	\centering
	\includegraphics[scale=0.8]{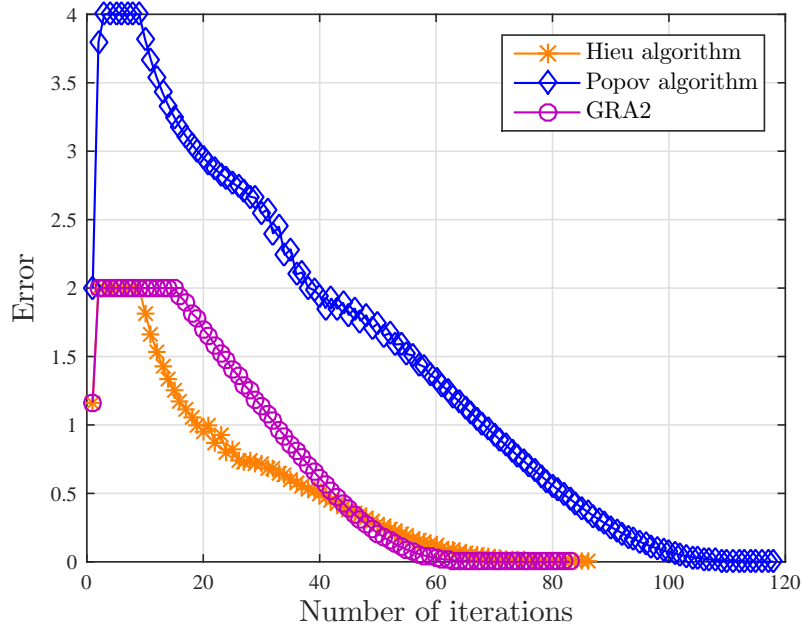}
	\caption{Comparison of three algorithms in Example \ref{vd2} with $x^0=\frac{1}{200}(\sin(-3t)+\cos(-10t))$}\label{h1vd2}
\end{figure}
\begin{figure}[H]
	\centering
	\includegraphics[scale=0.8]{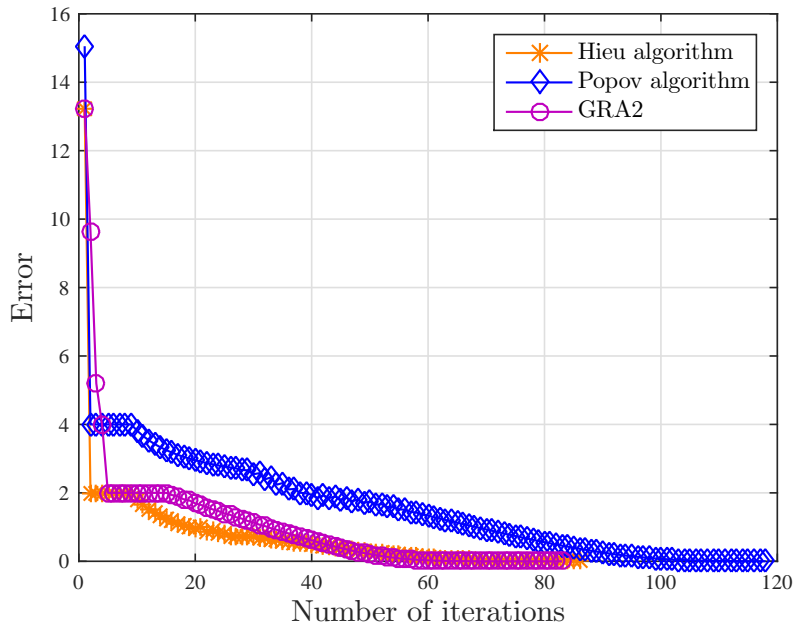}
	\caption{Comparison of three algorithms in Example \ref{vd2} with $x^0=\frac{1}{85}(t^3+1)e^{5t}$}\label{h1vd2}
\end{figure}
\end{example}

\section{Conclusions}
This paper deals with the convergence analysis and some numerical examples of the golden ratio algorithm for pseudomonotone equilibrium problems in Hilbert spaces. The proposed algorithm is an equilibrium version of a very recent algorithm introduced by Malitsky \cite{malitsky} (for variational inequalities). Moreover, the proposed algorithm is convergent under a weaker condition than the joint weak lower semicontinuity of the bifunction, assumed in several papers before. Numerical results show that the algorithm performs better than some existing methods. Note that, obtaining a result for Algorithm \ref{alg1} without using the condition (A5) seems to be more delicate and further investigations are necessary.

\section*{Acknowledgement}The author would like to thank Yura Malitsky for sending him the paper \cite{malitsky} and for many useful comments. %This research is funded by Vietnam National Foundation for Science and Technology Development (NAFOSTED) under Grant No.101.01-2017.08.

%\nocite{*}

\end{document}